\documentclass{amsart}
\usepackage{verbatim}
\usepackage{rotating} %for \includegraphix[angle=270]
\usepackage{amssymb}
\newtheorem{theorem}{Theorem}[section]
\newtheorem{proposition}[theorem]{Proposition}

\newtheorem{lemma}[theorem]{Lemma}

\newtheorem{cor}[theorem]{Corollary}

\newtheorem{conjecture}[theorem]{Conjecture}

\theoremstyle{plain}
\numberwithin{equation}{theorem}

\theoremstyle{remark}

\newtheorem{remark}[theorem]{Remark}

\newcommand{\Q}{{\mathbb Q}}

\newcommand{\cK}{{\mathcal K}}

\newcommand{\fo}{\mathfrak o}
\newcommand{\fq}{\mathfrak q}

\newcommand{\fp}{\mathfrak p}

\newcommand{\Kbar}{\overline K}
\newcommand{\kbar}{\overline{k}}

\DeclareMathOperator{\Orb}{Orb}

\DeclareMathOperator{\rad}{rad}

\newcommand{\Qbar}{\bar{\Q}}

\DeclareMathOperator{\Gal}{Gal}
\DeclareMathOperator{\Disc}{Disc}
\DeclareMathOperator{\Aut}{Aut}
\DeclareMathOperator{\N}{N}

\newcommand{\bP}{{\mathbb P}}
\newcommand{\bZ}{{\mathbb Z}}

\newcommand{\bC}{{\mathbb C}}

 \newcommand{\bQ}{{\mathbb Q}}

\newcommand{\lra}{\longrightarrow}

\begin{document}

\title{ABC implies primitive prime divisors in arithmetic dynamics}
\author{C.~Gratton, K.~Nguyen, and T.~J.~Tucker}

\address{
Chad Gratton\\
Department of Mathematics\\
Hylan Building\\
University of Rochester\\
Rochester, NY 14627
}

\email{grattonchad@gmail.com}

\address{
Khoa Nguyen \\
Department of Mathematics\\
University of California\\
Berkeley, CA 94720 
}

\email{khoanguyen2511@gmail.com}

\address{
Thomas Tucker\\
Department of Mathematics\\
Hylan Building\\
University of Rochester\\
Rochester, NY 14627
}

\email{thomas.tucker@rochester.edu}

\begin{abstract}
Let $K$ be a number field, let $\varphi(x) \in K(x)$
be a rational function of degree $d> 1$, and let $\alpha \in K$ be a
wandering point such that $\varphi^n(\alpha) \not= 0$ for all $n > 0$.  We
prove that if the $abc$ conjecture holds for
$K$, then for all but finitely many positive integers $n$, there is a prime $\fp$ of $K$ such
that $v_\fp( \varphi^n(\alpha)) > 0$ and $v_\fp(\varphi^m(\alpha)) \leq 0$ for all positive
integers $m < n$.  We prove the same result unconditionally for
function fields of characteristic 0 when $\varphi$ is not isotrivial.  
\end{abstract}

\thanks{The third author was partially supported by NSF Grants
  DMS-0854839 and DMS-1200749}

\maketitle

\section{Introduction}

Let $K$ be a number field or function field, let $\varphi(x) \in K(x)$
be a rational function of degree $d> 1$, and let $\alpha \in K$.  We
denote the $n$-iterate of $\varphi$ as $\varphi^n$.  It is often the
case that for all but finitely many $n$, there is a prime that is a
divisor of $\varphi^n(\alpha)$ that is not a divisor of
$\varphi^m(\alpha)$ for any $m < n$.  This problem was first
considered by Bang \cite{Bang}, Zsigmondy \cite{Z}, and Schinzel
\cite{Schinzel} in the context of the multiplicative group.  More
recently, many authors have considered the problem in other cases.
Most of these results apply either when 0 is preperiodic under
$\varphi$ (see \cite{FG, IS}, for example) or when 0 is a ramification
point of $\varphi$ (see \cite{DH, Krieger, Rice}).  In this paper, we
show that similar results will hold in close to full generality, assuming the
$abc$-conjecture of Masser-Oesterl\'e-Szpiro for number fields.  Our
result also holds unconditionally over characteristic 0 function fields, where the
$abc$ conjecture is a theorem of Mason \cite{Mason}.  We are not
however, able to derive it directly from the $abc$ conjecture in this
case, because of the absence of Belyi maps (see Lemma~\ref{1}) over
function fields; our proof requires a more difficult theorem of
Yamanoi \cite{Yam} conjectured by Vojta in \cite{Vojta-ABC}.

We will say that a field $K$ is an {\bf $abc$-field} if $K$ is a
number field satisfying the $abc$-conjecture \cite{Vojta-Diophantine-SV} or a characteristic zero function field of transcendence
degree 1.  We define the {\bf orbit $\Orb_\varphi(\alpha)$} of a
point $\alpha$ under a map $\varphi$ to be $\Orb_\varphi(\alpha) =
\bigcup_{i=1}^\infty \{ \varphi^i(\alpha) \}$.  The most general
results here are most naturally stated in terms of the canonical
height $h_\varphi$ of Call and Silverman \cite{CS} (see
\eqref{candef}, \eqref{can1}, \eqref{can2}, and \eqref{can3} for its
definition and a few of its basic properties).

With this notation and terminology, the main theorem of our paper is the
following.   

\begin{theorem}\label{main1}
  Let $K$ be an $abc$ field, let $\varphi \in K(x)$ have degree $d
  > 1$, and let $\alpha \in K$ be a point such that
  $h_\varphi(\alpha) > 0$ and $0 \notin \Orb_\varphi(\alpha)$.
  Suppose that $\varphi(z) \not= c z^{\pm d}$ for any
  $c \in K$.  Then for all but finitely many positive
  integers $n$, there
  is a prime $\fp$ of $K$ such that $v_\fp(\varphi^n(\alpha)) > 0$ and
  $v_\fp(\varphi^m(\alpha)) \leq 0$ for all positive integers $m < n$.
\end{theorem}

We will say that
$\varphi$ is {\bf dynamically ramified} if there at most are finitely many points $\gamma$ such that
$\varphi^n(\gamma)= 0$ and $e_{\varphi^n}(\gamma/0) = 1$ for some
$n$, where $e_{\varphi^n}(\gamma/0)$ is the ramification index of
$\varphi^n$ at $\gamma$ over 0.  

\begin{theorem}\label{main2}
  Let $K$ be an $abc$ field, let $\varphi \in K(x)$ have degree
  $d > 1$, and let $\alpha \in K$ have $h_\varphi(\alpha) > 0$
  and that $0 \notin \Orb_\varphi(\alpha)$.
  Suppose $\varphi$ is not dynamically ramified.  Then for
  all but finitely many positive integers $n$, there is a prime $\fp$ of $K$ such that
  $v_\fp(\varphi^n(\alpha)) = 1$ and $v_\fp(\varphi^m(\alpha)) \leq 0$
  for all positive integers $m < n$.   
\end{theorem}

In fact, Theorem~\ref{main2} can never hold for any $\varphi$ that {\it is}
dynamically ramified; see Remark~\ref{dyn}.  This is most easily seen
in the case of maps such as $\varphi(x) = (x-a)^2$,  which have the
property that $\varphi^n(\alpha)$ is always a perfect square because
$\varphi^n$ itself is a perfect square in the field of rational
functions.  
 
Theorem~\ref{main2} shows that the $abc$-conjecture implies what Jones
and Boston call the ``Strong Dynamical Wieferich Prime Conjecture''
\cite[Conjecture 4.5]{JB} .  Silverman \cite{SilW} had earlier shown
that the $abc$-conjecture implies a logarithmic lower bound on the
growth of the number of Wieferich primes; a Wieferich prime is a prime $p$
for which $2^{p-1} \not\equiv 1 \pmod{p^2}$.

In keeping with the terminology of \cite{IS, Schinzel}, we say that
$\fp$ is a {\bf primitive prime factor} of $\varphi^n(\alpha)$ if
$v_\fp(\varphi^n(\alpha)) > 0$ and $v_\fp(\varphi^m(\alpha)) \leq 0$
for all $m < n$.  We say that $\fp$ is a {\bf square-free primitive
 prime factor} if $v_\fp(\varphi^n(\alpha)) = 1$ and
$v_\fp(\varphi^m(\alpha)) \leq 0$ for all $m < n$.  Then
Theorem~\ref{main1} says that $\varphi^n(\alpha)$ has a primitive
prime factor for all but finitely many $n$ while Theorem~\ref{main2}
says that $\varphi$ has a square-free primitive prime factor for all but
finitely many $n$.

Theorems \ref{main1} and \ref{main2} may also be stated in terms of
{\bf wandering} $\alpha$.  We say that $\alpha$ is wandering if $\varphi^n(\alpha) \not= \varphi^m(\alpha)$ for all $n > m > 0$;
this is equivalent to saying that $\Orb_\varphi(\alpha)$ is infinite. It follows
immediately from Northcott's theorem that $h_\varphi(\alpha) \not= 0$ if
and only if $\alpha \in K$ is  wandering for $\varphi \in K(x)$, where $K$
is a number field and $\deg \varphi > 1$ (see \cite{CS}).   
By work of Benedetto and Baker \cite{BenFin, BakerF}, one has the same
result for non-isotrivial rational functions over a function
field. A rational function over a function field $K$ is said to be
{\bf isotrivial} if it cannot be defined over a finite extension of the field of constants of $K$,
up to change of coordinates; more precisely we say that $\varphi$ is
isotrivial if there exists $\psi \in \Kbar(x)$ of degree 1 such that
$(\psi^{-1} \circ \varphi \circ \psi) \in \kbar(x)$, where $\psi^{-1}$ is
the compositional inverse of $\psi$ (i.e., $\psi^{-1}(\psi(x)) = x$ in
$\Kbar(x)$).

Baker's result says that if $K$ is a function field and $\varphi \in
K(x)$ is a non-isotrivial map with $\deg \varphi > 1$, then a point
$\alpha \in K$ is wandering if and only if $h_\varphi(\alpha) \not= 0$.  

Thus, the following are immediate corollaries of Theorem \ref{main1}
and \ref{main2}.  
 
\begin{cor}\label{cor1}
  Let $K$ be an $abc$ field, let $\varphi \in K(x)$ have degree $d
  > 1$, and let $\alpha \in K$ be a wandering point such
  that $0 \notin \Orb_\varphi(\alpha)$.  Suppose that $\varphi \not= c
  z^{\pm d}$ for any $c \in K$ and that $\varphi$ is
  non-isotrivial if $K$ is a function field. Then for all but finitely
  many positive integers $n$, there is a prime $\fp$ of $K$ such that
  $v_\fp(\varphi^n(\alpha)) > 0$ and $v_\fp(\varphi^m(\alpha)) \leq 0$
  for all  positive integers $m < n$.
\end{cor}

\begin{cor}\label{cor2}
  Let $K$ be an $abc$ field, let $\varphi \in K(x)$ have degree
  $d >  1$, and let $\alpha \in K$ be a wandering point such
  that $0 \notin \Orb_\varphi(\alpha)$.  Suppose $\varphi$ is not
  dynamically ramified and that $\varphi$ is
  non-isotrivial if $K$ is a function field.  Then for all but
  finitely many positive integers $n$, there is a prime $\fp$ of $K$ such that
  $v_\fp(\varphi^n(\alpha)) = 1$ and $v_\fp(\varphi^m(\alpha)) \leq 0$
  for all positive integers $m < n$.
\end{cor}

The strategy of the proofs of Theorems \ref{main1} and \ref{main2} is
fairly simple.  First, we show, in Propositions \ref{Roth-number} and
\ref{Roth-function}, that if $F$ is a polynomial of reasonably high degree
without repeated roots, then for any $\gamma$ of large height,
the product of the distinct prime factors of $F(\gamma)$ is large,
assuming the $abc$-conjecture in the number field case.  We then apply
this to an appropriate factor $F$ of the numerator of a power
$\varphi^i$ of $\varphi$, after proving, in Proposition \ref{old},
that the product of the distinct factors of
$\prod_{\ell=1}^{n-1}\varphi^\ell(\alpha)$ that are also factors of
$F(\varphi^{n-i}(\alpha))$ must be very small.  With finitely many
exceptions, any prime that divides $F(\varphi^{n-i}(\alpha))$ also
divides $\varphi^n(\alpha)$, so $\varphi^n(\alpha)$ must then have a
factor that is not a factor of $\varphi^m(\alpha)$ for any $m < n$.

An outline of the paper is as follows. We begin by setting our
notation and terminology in Section~\ref{prelim}.  In Section~\ref{number} we
modify a result of Granville \cite{Granville} that enables
us to say, roughly, that polynomials without repeated factors take on
``reasonably square-free'' values in general, assuming the
$abc$-conjecture; this is Proposition~\ref{Roth-number}.  Then, in Section~\ref{function}, we derive the same
result for function fields, unconditionally, using recent work of Yamanoi \cite{Yam}; this
is Proposition \ref{Roth-function}.  
This enables us to give a proof of our main results in
Section~\ref{main}, using Proposition \ref{old}.  We end with some
applications of Theorem~\ref{main2} to iterated Galois groups, in
Section~\ref{application}.  

\begin{remark}
  When $0$ is in $\Orb_\varphi(\alpha)$ and $\alpha$ is wandering,
  there is a unique $M$ such that $\varphi^{M}(\alpha)=0$.  Hence, Theorems \ref{main1} and \ref{main2}
  still hold if we impose the additional condition $m \not= M$ on the
  positive integers $m < n$ in the statements of these theorems.    
\end{remark}

{\bf Acknowledgments.}  The authors would like to thank Xander Faber, Dragos Ghioca,
Andrew Granville, Patrick Ingram, Rafe Jones, Joseph Silverman, and
Paul Vojta for many helpful conversations.  This paper was written
while the authors were visiting ICERM in Providence and it is our
pleasure to thank ICERM for its hospitality.

\section{Preliminaries}\label{prelim}

We set the following:
\begin{itemize}
\item $K$ is a number field or function field of characteristic 0;
\item if $K$ is a function field, we let $k$ denote its field of
  constants;
\item $\fp$ is a finite prime of $K$
\item $k_\fp$ is the residue field of $\fp$;
\item if $K$ is a number field, then we let $\N_\fp = \frac{\log
    (\#k_\fp)}{[K:\bQ]}$;
\item if $K$ is a function field, then we let $\N_\fp = [k_\fp:k]$;  
\item $\varphi \in K(x)$ is a rational function of degree $d > 1$.
\end{itemize}

All of this is completely standard with one exception: the quantity
$\N_\fp$ has been normalized in the case of number fields.  We divide by $[K:\bQ]$ in
our definition that we can use the same proofs (without reference to
possible normalization factors) for number fields and function fields
in Section~\ref{main}.  

When $K$ is a number field, we let $\fo_K$ denote the ring of
algebraic integers of $K$ as usual.  When $K$ is a function field, we
choose a prime $\fq$, and let $\fo_K$ denote the set $\{ z \in K \; |
\; \text{$v_\fp(z) \geq 0$ for all primes $\fp \not= \fq$ in $K$} \}$.  

If $K$ is a function field, the height of $\alpha \in K$ 
is
\begin{equation*}
h(\alpha) = - \sum_{\text{ primes $\fp$ of $K$ }} \min
(v_\fp(\alpha), 0) \N_\fp.
\end{equation*} 
If $K$ is a number field,  the height of $\alpha \in K$ 
is
\begin{equation}\label{ht-number}
h(\alpha) = - \sum_{\text{ primes $\fp$ of $\fo_K$ }}   \min
(v_\fp(\alpha), 0) \N_\fp +
\frac{1}{[K:\bQ]}  \sum_{\sigma: K \hookrightarrow \bC}
  \max (\log |\sigma(\alpha)|, 0). 
\end{equation}
(Note that the $\sigma: K \hookrightarrow \bC$ is simply all maps from
$K$ to $\bC$; in particular, we do not identify complex conjugate
embeddings in any way.)
In either case, the product formula gives the inequality
\begin{equation}\label{ineq}
\sum_{v_\fp(\alpha) > 0} v_\fp(\alpha) \N_\fp \leq h(\alpha).
\end{equation} 

We will work with the canonical height $h_\varphi$, which is defined
as 
\begin{equation}\label{candef}
h_\varphi(z) = \lim_{n \to \infty} \frac{h(\varphi^n(z)}{d^n}.
\end{equation}
The convergence of the right-hand side follows from a telescoping
series argument due to Tate.  The canonical height has the following important
properties:
\begin{equation}\label{can1}
\text{ $h_\varphi(\varphi(z)) = d h_\varphi(z)$ for all $z \in K$;}
\end{equation}
\begin{equation}\label{can2}
\text{ there is a constant $C_\varphi$ such that $|h(z) - h_\varphi(z)| 
  < C_\varphi$ for all $z \in K$.} 
\end{equation}
It follows immediately from \eqref{can1} and \eqref{can2} that
\begin{equation}\label{can3}
h_\varphi(\alpha) \not= 0 \Longleftrightarrow \lim_{s \to
    \infty} h(\varphi^s(\alpha)) =  \infty.  
\end{equation}

We refer the readers to the work of Call and Silverman \cite{CS} for details on the proofs of the
various properties of $h_\varphi$.

We say that a point $\alpha$ is {\bf preperiodic} if there exist $n >
m > 0$ such that $\varphi^m(\alpha) = \varphi^n(\alpha)$; we will
say that $\alpha$ is {\bf periodic} if there is an $n > 0$ such that
$\varphi^n(\alpha) = \alpha$.  Note that a point is wandering if and only
if it is not preperiodic.

We write $\varphi(x) = P(x)/Q(x)$ for $P, Q \in \fo_K[x]$ having no common
roots in $\Kbar$.  Then we may write $\varphi^i(x) = P_i(x)/Q_i(x)$,
where $P_i$ and $Q_i$ are defined recursively in terms of $P$ and $Q$.
This is most easily explained by passing to homogenous coordinates.
We let $p(x,y)$ and $q(x,y)$ be the degree $d$ homogenizations of $P$
and $Q$ respectively.  Then we may define $p_i(x,y) = p(p_{i-1}(x,y),
q_{i-1}(x,y))$ and $q_i(x,y) = q(p_{i-1}(x,y), q_{i-1}(x,y))$.
Letting $P_i = p_i(x,1)$ and $Q_i = q_i(x,1)$ then gives our $P_i$ and
$Q_i$.  We will say that $\fp$ is a prime of {\bf good reduction} if
$P(x)$ and $Q(x)$ have no common root modulo $\fp$ and the polynomials
$p(1,y)$ and $q(1,y)$ have no common roots modulo $\fp$.  When
$\fp$ is a prime of good reduction, $\varphi$ induces a well-defined
map from $k_\fp \cup \infty$ to itself.  To describe this, let $r_p$
be the reduction map $r_p: K \lra k_\fp \cup \infty$ given by $r_p(z)
= z \pmod {\fp}$ if $v_\fp(z) \geq 0$ and $r_p(z) = \infty$ if
$v_\fp(z) < 0$.  Then letting $\varphi(r_p(z)) = r_p(\varphi(z))$
defines a well-defined map on residue classes and thus gives the desired
map. We will make use of this in Proposition~\ref{old}.

When $K$ is a function field, we say that $\varphi$ is isotrivial if
$\varphi = \sigma \psi
\sigma^{-1}$ for some $\sigma \in \Kbar(x)$ with $\deg \sigma = 1$ and
some $\psi \in \kbar(x)$, where $k$ is the field of constants in $K$.
Here $\sigma^{-1}$ is the compositional inverse of $\sigma$; we have
$\sigma(\sigma^{-1}(x)) = \sigma^{-1}(\sigma(x)) = x$ in
the field $\Kbar(x)$.  

Finally, a few words on notation.  Throughout this paper,  a {\bf finite
  set}  may be empty, so when we say ``there exist finitely many''
that simply means that there are not infinitely many.  The zeroth iterate
of any map is taken to be the identity; in particular, $\varphi^0(x) =
x$.

\section{Roth-$abc$ for number fields}\label{number}
The main result of this Section, Proposition~\ref{Roth-number}, is a
direct translation of \cite[Theorem 5]{Granville} into the more
general setting of number fields.  Following Granville, we refer to
this as a ``Roth-abc'' type result, because it can be interpreted as
a strengthening of Roth's theorem \cite{Roth}  (in particular the $- 2 - \epsilon$
here plays the same role as the the $2+\epsilon$ in Roth's theorem).
The techniques are the same as those of \cite{Granville}.  We include
a full proof for the sake of completeness.  The methods here are also
quite similar to those of \cite{Elkies} (see especially page 105).

Let $K$ be a number field. We will be using a version of the ``$abc$-Conjecture for Number
Fields''.    Recall our definition of $h(z)$ for $z \in K$ from
\eqref{ht-number}.  For $n \geq 2$, we may extend  this definition to
an  $n$-tuple  $(z_1, \dots, z_n) \in K^n
\setminus \{ (0,\dots,0) \}$ by letting
\begin{equation}
\begin{split}
  h(z_1,\dots, z_n) =  - & \sum_{\text{ primes $\fp$ of $\fo_K$ }}   \min
  (v_\fp(z_1),\dots,  v_\fp(z_n)) \N_\fp  \\
& +  \frac{1}{[K:\bQ]}  
\sum_{\sigma: K \hookrightarrow \bC}
  \max (\log |\sigma(z_1)|, \dots, \log |\sigma(z_n)|). 
\end{split}
\end{equation}

Note that when $z_1 \not= z_2$, we have $h(z_1,z_2) = h(z_1/z_2, 1) =
h(z_1/z_2)$.

 For any $(z_1, \dots, z_n) \in (K^*)^n$, we define 
\[ I(z_1, \dots, z_n) = \{ \text{primes $\fp$ of $\fo_K$} \; | \;
\text{$v_\fp(z_i) \not= v_\fp(z_j)$ for some $1 \leq i,j \leq n$}\} 
\]
and let
\begin{displaymath}
\rad(z_1,\dots, z_n) = \sum_{\mathfrak p \in
  I(z_1,\dots, z_n)} \N_\fp
\end{displaymath}

With all of this notation set, the
$abc$-Conjecture for number fields says the following.
\begin{conjecture}
For any $\epsilon > 0$, there exists a constant $C_{K,
  \epsilon} > 0$ such that for all $a, b, c \in K^*$ satisfying $a + b  = c$, we have
\begin{displaymath}
h(a, b, c) < (1 + \epsilon) (\rad(a, b, c)) + C_{K,\epsilon}.
\end{displaymath}
\end{conjecture}

Following Granville \cite{Granville}, we start by proving a homogeneous form of
Roth-$abc$.   Let $S$ be a finite set of finite primes of $K$.  We will say that a pair $(z_1, z_2) \in
\fo_K$ is in {\bf $S$-reduced form} if they have no common prime factors
outside of $S$, that is $\min(v_\fp(z_1), v_\fp(z_2)) = 0$ for
all $\fp \notin S$.

We will use a well-known result of Belyi \cite{Belyi}.  

\begin{lemma}\label{1} Given any homogeneous $f(x, y) \in
  K[x, y]$, we can determine homogeneous
  polynomials $a(x, y), b(x, y), c(x, y) \in \fo_K[x, y]$, all of
  degree $D \geq 1$, with bounded common factors, where $a(x, y)b(x,
  y)c(x, y)$ has exactly $D + 2$ non-proportional linear factors (over
  $\Kbar$), which
  include all the factors of $f(x, y)$, and $a(x, y) + b(x, y) = c(x,
  y)$.
\end{lemma}

We may then prove the following.  

\begin{proposition}\label{homog} 
Let $f(x,y) \in \fo_K[x,y]$ be a homogeneous polynomial of degree 3 or more without
repeated factors and let $S$ be a a finite set of finite primes in
$K$.  Let $\epsilon > 0$ and let $S$ be a finite set of finite places
of $K$.  Then
\begin{displaymath}
 (\deg f - 2 - \epsilon)(h(z_1, z_2)) \leq \left(\sum_{v_\fp(f(z_1, z_2)) > 0} \N_{\fp}\right) + O(1)
\end{displaymath}
for all $(z_1, z_2) \in \fo_K$ in $S$-reduced form.  
\end{proposition}

\begin{proof}
  We begin by applying Lemma \ref{1} to obtain $a(x, y), b(x, y), c(x,
  y) \in \fo_K[x, y]$ of degree $D$ where $a(x, y)b(x,
  y)c(x, y)$ has exactly $D + 2$ non-proportional linear factors (over $\Kbar$), which
  include all the factors of $f(x, y)$, and $a(x, y) + b(x, y) = c(x,
  y)$.  Write the product of the factors of $a(x, y)b(x, y)c(x, y)$ as $f(x, y)g(x, y)$.

  Then applying the $abc$-Conjecture for number fields, we obtain
  \[(1 - \epsilon/D)(h(a(z_1, z_2), b(z_1, z_2))) \leq
  \left(\sum_{\mathfrak p \in I(a(z_1, z_2), b(z_1, z_2), c(z_1,
    z_2))}\N_{\mathfrak p}\right) + O(1).\] 

  Now, $a$, $b$, and $c$ are coprime and $(z_1, z_2)$ is in $S$-reduced
  form, so, possibly after enlarging $S$, have a finite set $S$ of primes, depending only on $a$,
  $b$, $c$ and $K$ such that $\fp \in I(a(z_1, z_2), b(z_1,
  z_2), c(z_1, z_2))$ if and only $v_\fp(a(z_1,z_2)
  b(z_1,z_2) c (z_1,z_2)) > 0$ for all $\fp \notin S$.
  Since $a(x,y)b(x,y)c(x,y)$ has the same factors as $f(x,y) g(x,y)$,
  we therefore have
\[ \left(\sum_{\mathfrak p \in I(a(z_1, z_2), b(z_1, z_2), c(z_1,
    z_2))}\N_{\mathfrak p}\right) \leq
  \left(\sum_{v_\fp(f(z_1,z_2)) > 0} \N_{\fp}\right) + \left(\sum_{v_\fp(g(z_1,z_2)) > 0} \N_{\fp}\right) + \textrm{ }O(1), 
\]
so 
\begin{equation}\label{multi}
(1 - \epsilon/D)(h(a(z_1, z_2), b(z_1, z_2))) \leq  \left(\sum_{v_\fp(f(z_1,z_2)) > 0} \N_{\fp}\right) + \left(\sum_{v_\fp(g(z_1,z_2)) > 0} \N_{\fp}\right) + \textrm{ }O(1).
\end{equation}

By basic properties of height functions, we have
\begin{equation*}
\sum_{v_\fp(g(z_1, z_2)) > 0} \N_{\mathfrak p} \leq
h(g(z_1, z_2)) \leq  (D + 2 - \deg f)(h(z_1,
z_2)) + O(1),
\end{equation*}
since $g$ has degree $D+2 - \deg f$.    Similarly, we have
\[ h(a(z_1, z_2), b(z_1, z_2)) + O(1) \geq D(h(z_1, z_2)). \]
Substituting these inequalities into \eqref{multi} gives 
\[  (\deg f - 2
- \epsilon)(h(z_1, z_2)) \leq \left(\sum_{v_\fp(f(z_1, z_2)) > 0}
  \N_{\fp}\right) + O(1)\] as desired.
\end{proof}

\begin{proposition}\label{Roth-number}
Let $F(x) \in \fo_K[x]$ be a polynomial of degree 3 or more without
repeated factors.  Then, for any $\epsilon > 0$, there is a constant
$C_{f, \epsilon}$ such that 
\[
\sum_{v_\fp(f(z)) > 0} \N_\fp \geq (\deg F - 2 - \epsilon) h(z)  +
C_{f, \epsilon}
\]
for all $z \in K$.  
\end{proposition}
\begin{proof}
  By the Minkowski's theorem on the class group, there is a computable
  set $S$ of finite primes $\fp$ of $K$ such that the set of
  $\fo_{K,S}$ is a principal ideal domain where $\fo_{K,S}$ is the
  usual set of $S$-integers given by
$\fo_{K, S} = \{x
\in \fo_K \; | \; \text{ $|x|_{\mathfrak p} \leq 1$ for all $\mathfrak
  p \notin S$} \} $.
Thus we may write $z = z_1/z_2$ where
$(z_1, z_2)$ is in $S$-reduced form for a computable finite $S$ of finite
primes depending only on $K$.

Let $g(x,y)$ be the homogenization of $F(x)$ so that
  $g(x,1) = F(x)$ and $g(z_1,z_2) = z_2^{\deg f} F(z_1)$.  Let
  $f(x,y) = y g(x,y)$.  Let $T_1$ be a set of primes such that $z_1$
  and $z_2$ have no common zeroes outside $T_1$ and let $T_2$ be the set
  of primes such that $|a_n|_p \not= 1$ for some nonzero coefficient
  $a_n$ of $F$ (note that $T_1$ and $T_2$ are finite and depend only on
  $K$ and $F$).  Then for all $\fp \notin T_1 \cup T_2$, we have
  $v_\fp(f(z)) \not= 0$ if and only if $v_\fp(f(z_1, \beta)) > 0$.
  Thus, we have
\[ \left(\sum_{v_\fp(F(z)) \not= 0} \N_{\fp}\right) + O(1) \geq
\sum_{v_\fp(f(z_1,z_2)) > 0} \N_{\fp}.  \]
Since $h(z) = h(z_1,z_2)$ and $\deg f =  \deg F + 1$,
applying Proposition \ref{homog} gives
\begin{equation}\label{last}
\begin{split}
\left(\sum_{v_\fp(f(z_1, z_2)) > 0} \N_{\fp}\right) + O(1) \geq
(\deg F - 1 - \epsilon)(h(z_1, z_2)).
\end{split}
\end{equation} 
For $\fp \notin T_1 \cup T_2$, we have $v_\fp(F(z)) < 0$ exactly when
$v_\fp(z) < 0$, so 
\[\sum_{v_\fp(f(z)) > 0}
\N_{\fp}  \leq  h(z) + O(1).\]
 Thus,  we have a constant $C_{f,\epsilon}$ such that
$\sum_{v_\fp(f(z)) > 0} \N_\fp \geq (\deg F - 2 - \epsilon) h(z)  +
C_{f, \epsilon}$,
as desired.  
\end{proof}

\section{Roth-$abc$ for function fields}\label{function}

Using Yamanoi's theorem \cite[Theorem 5]{Yam} which establishes
a conjecture of
Vojta for function fields (see also \cite{Gasbarri, McQ2}), we obtain a
function field analog of Proposition~\ref{Roth-number}. Note that a more general implication is proved by Vojta in \cite{Vojta-ABC}, see also
\cite[p. 202]{Vojta-CIME}. In the special case needed here, 
 we include a short proof for the sake of completeness.

Let $V$ be a curve over a function field $K$, and let $\beta \in
V(\Kbar)$.  Then we define 
$$ d(\beta) = \frac{1}{[K(\beta):K]}\sum_{\text{ primes $\fp$ of $K$ }}
(v_\fp(\Delta_{K(\beta)/K}))$$
where $\Delta_{K(\beta)/K}$ is the relative discriminant 
of the extension $K(\beta)/K$.

Since we are working over a function field of characteristic 0 (so
that all ramification is tame), we may
use the definition
\[ d(\beta) = \frac{1}{[K(\beta):K]} \sum_{\text {primes $\fq$ of
      $K(\beta)$}} (e(\fq / (\fq \cap \fo_K)) - 1) \N_\fq \]
where $e(\fq / (\fq \cap \fo_K))$ is the ramification index of $\fq$ over
$\fq \cap \fo_K$.  

Let $\cK_V$ be the usual canonical divisor on $V$, and let $h_{\cK_V}$
be a height function for $\cK_V$.  Yamanoi \cite{Yam} proves the following
result, sometimes called the Vojta $(1+\epsilon)$-conjecture.

\begin{theorem}\label{Yama}(Yamanoi)
Let $K$ be a function field, let $V$ be a curve over $K$, let $M$
be a positive integer, and let $\epsilon > 0$.  Then there is a
constant $C_{M,\epsilon}$ such that for all $\beta \in V(\Kbar)$ with
$[K(\beta):K] \leq M$, we have
\begin{equation}\label{ep}
h_{\cK_V}(\beta) \leq (1 + \epsilon) d(\beta) + C_{M, \epsilon}.
\end{equation}
\end{theorem}

We will use Theorem~\ref{Yama} to prove
Proposition~\ref{Roth-function}, the function field analog of
Proposition~\ref{Roth-number}.  To do this, we first introduce a little
information about height functions and divisors. 

The divisor $\cK_V$ has degree $2g_V - 2$ where $g_V$ is the genus of
$V$.  By the standard theory of heights on curves (see
\cite[Proposition 1.2.9]{Vojta-Diophantine-SV}), for example), 
if $D$ is any
ample divisor, and $D'$ is an arbitrary divisor, we have
\begin{equation}\label{equiv}
	\lim_{h_{D}(z)\rightarrow\infty}\frac{h_{D'}(z)}{h_{D}(z)}=\frac{\deg D'}{\deg D}. 
\end{equation}

Now, let $\pi: V \lra \bP^1$ be a nonconstant map on a curve.  Suppose
that $\pi(\beta) = z$ for $z \in \bP^1(\Kbar)$.  The usual height $h(z)$
comes from a degree 1 divisor on $\bP^1$ which pulls back to a degree
$\deg \pi$ divisor on $V$.  Furthermore if $\pi(\beta) \in \bP^1(K)$,
then $[K(\beta): K] \leq \deg \pi$.   Thus, Theorem~\ref{Yama} and \eqref{equiv}
imply that for any $\epsilon' > 0$, we have
\begin{equation}\label{use}
(1 - \epsilon') \frac{2g_V - 2}{\deg \pi} h(\pi(\beta)) \leq
d(\beta) + O_{\epsilon'}(1) 
\end{equation}
for all $\beta \in V(\Kbar)$ such that $\pi(\beta) \in \bP^1(K)$.

We will use this to prove a function field analog of Proposition
\ref{Roth-number}.

\begin{proposition}\label{Roth-function}
  Let $K$ be a function field and let $F(x) \in K[x]$ be a
  polynomial of degree 3 or more without repeated factors.
  Then, for any $\epsilon > 0$, there is a constant $C_{F, \epsilon}$
  such that
\begin{equation}\label{func}
\sum_{v_\fp(f(z)) > 0} \N_\fp \geq (\deg F - 2 - \epsilon) h(z)  +
C_{F, \epsilon}
\end{equation}
for all $z \in K$.  
\end{proposition}
\begin{proof}
% Note that extending the field of constants $k$ of $K$ to its algebraic
% closure will not affect either side of \eqref{func}.  Thus, we may
% assume that $k$ is algebraically closed, so that $k_\fp = k$ for all
% primes $\fp$ of $K$.  

For each $n>0$, let $V_n$ be the nonsingular projective model over $K$
of $y^n=F(x)$. To calculate the genus $g_n$ of $V_n$, we use the
morphism $\pi : V_n
{\longrightarrow}\bP^1$ given by projection onto the $x$-coordinate;
that is, $\pi(x,y) = x$.  

From now on, we choose $n$ such that it is relatively prime $\deg F$.
This makes the above morphism totally ramified at zeroes and poles of
of $F$ and unramified everywhere else.  Since $F$ has a single pole at
the point at infinity along with $\deg F$ zeros, and $\pi$ has degree
$n$, the Riemann-Hurwitz theorem gives
  \begin{equation}\label{1stR-H}
  2g_{n}-2 =    (n-1)(\deg F + 1) - 2n = n(\deg F - 1) - (\deg F + 1).
\end{equation}     

Suppose that $\pi(\beta) = z \in K$.  Then  \eqref{use} and
\eqref{1stR-H} together give
\[ (1 -\epsilon') \left(\deg F - 1 - \frac{\deg F + 1}{n} \right) h(z)
\leq
d(\beta) + O_{\epsilon',n}(1) \]
Let $\epsilon > 0$.  Choosing sufficiently large $n$ and
sufficiently small $\epsilon'$ yields
\begin{equation}\label{with-poles}
 (\deg F - 1 - \epsilon) h(z) \leq d(\beta) + O_{n,\epsilon}(1).
\end{equation}

Now, $K(\beta) = K(\sqrt[n]{F(z)})$, which can only ramify over a
prime $\fp$ when $v_\fp(F(z)) \not= 0$.  Since $e(\fq / (\fq \cap
\fo_K)) \leq n -1$, where $e(\fq / \fq \cap \fo_K)$ is the
ramification index of $\fq$ over $\fq \cap \fo_K$, we have $d(\beta)
\leq \sum_{v_\fp(F(z)) \not= 0} \N_p$.  When $v_\fp(F(z)) < 0$, either
$v_\fp(z) < 0$ or $v_\fp(a_i) < 0$ for some coefficient $a_i$ of
$F(z)$.  Since $F$ has only finitely many coefficients and each has
negative valuation at only finitely many primes, this means that
$\sum_{v_\fp(F(z)) < 0} \N_p \leq h(z) + O_F(1)$.  Hence,
\begin{equation}\label{almost}
 d(\beta) \leq \sum_{v_\fp(F(z)) > 0} \N_p  + h(z) + O_F(1).
\end{equation}
Combining \eqref{with-poles} with \eqref{almost} then gives
\eqref{func}.   
\end{proof}
 
\section{Proofs of main theorems}\label{main}

We begin with a proposition that allows us to control the size of
certain non-primitive factors of $\varphi^n(\alpha)$.  We choose a
polynomial factor $F$ of the numerator $P_i$ of $\varphi^i(z)$ and use the fact that, outside a
finite set of primes, we have $v_\fp(\varphi^n(\alpha)) > 0$ whenever
$v_\fp(F(\varphi^{n-i}(\alpha))) > 0$.  If $m < n$, the condition
\[\min(v_\fp(F(\varphi^{n-i}(\alpha))), v_\fp(\varphi^m(\alpha))) > 0\]
 forces
some root of $F$ to be periodic modulo $\fp$, with period at most
$n-m$.  If all of the roots of $F$ are non-periodic, then, for bounded
$n-m$, there are at most finitely many such $\fp$.  Thus, any $\fp$
such that $\min(v_\fp(F(\varphi^{n-i}(\alpha))), v_\fp(\varphi^m(\alpha))) > 0$ comes from either a bounded set or from a relatively low order
iterate $\varphi^\ell(\alpha)$ of $\alpha$.  Since $h(\varphi^\ell(\alpha))$ is very small
relative to $h(\varphi^n(\alpha))$ when $\ell$ is small relative to
$n$, this allows for a strong lower bound on the product of all such
$\fp$.
    
\begin{proposition}\label{old}
  Let $\delta > 0$, let $\alpha \in K$ such that $\lim_{s \to \infty}
  h(\varphi^s(\alpha)) = \infty$, and let $F$ be a factor of the numerator
  of $\varphi^i$ such that every root $\beta_j$ of $F$ is
  non-periodic and satisfies $\varphi^\ell(\beta_j) \not= 0$ for $\ell =
  0, \dots,i-1$.   Then there is a constant $C_\delta$ such that for
  all positive integers $n$, we have
\begin{equation}\label{oldeq}
\sum_{\fp \in Z} \N_\fp \leq \delta h(\varphi^n(\alpha)) + C_\delta, 
\end{equation}
where $Z$ is the set of primes $\fp$ such that $\min(v_\fp(\varphi^m(\alpha)),
v_\fp(F(\varphi^{n-i}(\alpha)))   ) > 0$ for some positive integer $m < n$. 
\end{proposition}
\begin{proof}
Let $L$ be a finite extension of $K$ over which $F$ splits completely
as $F(x) = a (x-\beta_1) \dots (x-\beta_s)$, for $\beta_j \in L$.
Then, for all but finitely many primes $\fp$ of $K$, we have
$v_\fp(F(z)) > 0$ if and only if $v_\fq(z - \beta_j) > 0$ for some
prime $\fq$ of $L$ with $\fp = \fq \cap \fo_K$.  Thus, it suffices to show
that for each $\beta_j$, there is a $C_\delta$ such that for all $n$
we have
 \begin{equation}\label{want}
\sum_{\fp \in Y} \N_\fp \leq \delta h(\varphi^n(\alpha)) + C_\delta, 
\end{equation}
where $Y$ is the set of primes $\fp$ such that $\min(v_\fq(\varphi^m(\alpha)),
v_\fq(\varphi^{n-i}(\alpha) - \beta_j)  ) > 0$ for some positive integer
$m < n$ and some prime $\fq$ of
$L$ with $\fq | \fp$.

Let $Y_1$ be the set of primes of $L$ at which $\varphi$ does not have good
reduction, as defined in Section~\ref{prelim}.   Write $P_i = FR$, and
let $Y_2$ be the finite set of primes $\fq$ at which some $|b_s|_\fq \not= 1$
for some nonzero coefficient of $F$ or $R$.  Then, for all $\fq$
outside of $Y_1 \cup Y_2$, we have $\varphi(z) \equiv 0 \pmod{\fq}$ whenever
$F(z) \equiv 0 \pmod{\fq}$.

If $\min(v_\fq(\varphi^m(\alpha)),
v_\fq(\varphi^{n-i}(\alpha) - \beta_j)  ) > 0$ for $n-i \leq m < n$, then
$v_\fq(\varphi^{n-m}(\beta_j)) > 0$.  The set $Y_3$ of primes for
which this can happen is therefore finite since  $\varphi^\ell(\beta_j) \not= 0$ for $\ell =
  0, \dots,i-1$. 
 
  For any $B$, the let $W_B$ be the set of primes outside $Y_1 \cup Y_2
  \cup Y_3$ such that $\min(v_\fq(\varphi^m(\alpha)),
  v_\fq(\varphi^{n-i}(\alpha) - \beta_j) ) > 0$ for some some positive
  integers $m$ and $n$ with $n - i > m > n- i - B$. If $\fq \in W_B$,
  then $\varphi^{m}(\alpha) \equiv \varphi^n(\alpha) \equiv 0
  \pmod{\fq},$ so 0 is in a cycle of period at most $n-m$ modulo
  $\fq$.  Since $\beta_j \equiv \varphi^{n-i}(\alpha) \equiv
  \varphi^{(n-i)-m}(0) \pmod{\fq}$, we see that $\beta_j$ is in the same
  cycle modulo $\fq$. This implies that $\beta_j$ has period $B+i$ or
  less modulo $\fq$.  Since $\beta_j$ is not periodic, there are only
  finitely many such $\fq$, so $W_B$ must be finite.  (Note that $\varphi$
  induces a well-defined map from $k_\fq \cup \infty$ to itself,
  because $\fq$ is a prime of good reduction for $\varphi$.)

  Note that $v_\fp(\varphi^l(\alpha)) > 0$ if and only if
  $v_\fq(\varphi^l(\alpha)) > 0$ for some $\fq | \fp$.  Let $Z_B$ be
  the set $\{\textrm{primes } \fp \in \fo_K \; | \; \fq | \fp \text {
    for some $\fq \in W_B$} \}$.  When $\fp \notin Z_B \cup Y_1 \cup
  Y_2 \cup Y_3$, we see then that if $\min(v_\fp(F(\varphi^{n -
    i}(\alpha))), v_\fp(\varphi^n(\alpha))) > 0$ then
  $v_\fp(\varphi^{m}(\alpha)) > 0$ for some positive integer $m \leq n
  -i - B$.  Since $Y_1$, $Y_2$, and $Y_3$ are finite, and $Z_B$ is
  finite for any positive integer $B$, see that for any $B$, there is
  a constant $C_B$ such that
\begin{equation}\label{B}
\sum_{\fp \in Y} \N_\fp \leq \sum_{\ell=1}^{n-i-B}
\sum_{v_\fp(\varphi^\ell(\alpha)) > 0} N_\fp + C_B\leq
\sum_{\ell=1}^{n - B - i} h(\varphi^\ell(\alpha)) + C_B
\end{equation}
where $Y$ is the set of primes $\fp$ where $\min(v_\fq(\varphi^m(\alpha)),
v_\fq(\varphi^{n-i}(\alpha) - \beta_j)  ) > 0$ for some positive integer $m < n$.   

It suffices then to show that then, that for any $\delta$, we have
\begin{equation}\label{want2}
  \sum_{\ell=1}^{n - B - i} h(\varphi^\ell(\alpha)) < 
  \delta(h(\varphi^n(\alpha)))  
\end{equation}
for all sufficiently large $n$.  We will use a
telescoping sum argument and the canonical height of Call and
Silverman \cite{CS} here.  Recall that by \eqref{can1}, we have
$h_\varphi(\varphi(z)) = d h_\varphi(z)$ for all $z \in K$ and
that by \eqref{can2}, there is a constant $C_\varphi$ such that $|h(z)
- h_\varphi(z)| < C_\delta$ for all $z \in K$.
 
Choose $B_\delta$ such that $1/d^{B_\delta +i} < \delta/4$ and  $d^{n}(h_\varphi(\alpha)) >
\frac{(n+1) C_\varphi}{\delta/2} $ for all $n > B_\delta$.  Then for
all $n > B_\delta$, we have 
\begin{equation}\label{m}
\begin{split}
\sum_{\ell=1}^{n-B-i} h(\varphi^\ell(\alpha)) & \leq \sum_{\ell=1}^{n-B-i} h_\varphi(\varphi^\ell(\alpha)) + n C_\varphi\\
& = \frac{1}{d^{B_\delta +i}} \sum_{r=0}^{n-B- i -1}
\frac{h_\varphi(\varphi^n(\alpha))}{d^r} + n C_\varphi \text{ (by \eqref{can1})}\\
& \leq   \left(\frac{1}{d^{B_\delta +i}} \sum_{r=0}^\infty\frac{1}{d^r} \right)
h_\varphi(\varphi^n(\alpha)) + n C_\varphi\\
& \leq  \frac{\delta}{2}  h_\varphi(\varphi^n(\alpha))  + n C_\varphi\\
& \leq  \frac{\delta}{2}  h(\varphi^n(\alpha)) + (n+1) C_\varphi
\text{ (by \eqref{can2})}\\
& \leq  \delta h(\varphi^n(\alpha)).
\end{split}
\end{equation}

Thus, \eqref{want2} holds, and our proof is complete.  
\end{proof}

We say that a point $\beta$ is exceptional if $\varphi^{-2}(\beta) =
\beta$.   The condition $\varphi(z) \not= c z^{\pm d}$ implies that 0 is not
exceptional.  

\begin{lemma}\label{lemnon}
If $\beta \in \Kbar$ is not exceptional, then $\varphi^{-3}(\beta)$ contains at least
two distinct points in $\bP^1(\Kbar)$.    
\end{lemma}
\begin{proof}
  If $\varphi^{-3}$ contains only one point, $\gamma$, then $\varphi$
  is totally ramified at $\gamma$, $\varphi(\gamma)$, and
  $\varphi^2(\gamma)$.  By Riemann-Hurwitz, $\varphi$ can have at most
  two totally ramified points, so this means that $\gamma$,
  $\varphi(\gamma)$, and $\varphi^2(\gamma)$ are not distinct, to we
  must have $\varphi^2(\gamma) = \gamma$, so $\gamma$ is exceptional.
  But then $\beta$ must be exceptional too.
\end{proof}

Now, we can give a quick proof of \ref{main1}.

\begin{proof}[Proof of Theorem \ref{main1}]
  There is an $i$ such that $P_i$ has a factor $F \in K[x]$ of degree
  4 (see Remark \ref{deg}) or more such that every root $\beta_j$ of $F$ is non-periodic
  and satisfies $\varphi^\ell(\beta_j) \not= 0$ for $\ell = 0,
  \dots,i-1$.  To see this note that since $\varphi^{-3}(0)$ contains
  two points, by Lemma \ref{lemnon}, at least one of these points is
  not periodic.  Taking the third inverse image of this point yields
  at least four non-periodic points; if one of these is the point at
  infinity, then three further inverse images yields eight points, at
 none of which is the point at infinity.  Let $i$ be the smallest integer such
  that $\varphi^i(z) = 0$ for these points $z$ (this $i$ is the same
  for all of them since they are all inverse images of the same
  non-periodic point), and let $F \in K[x]$ be a factor of $P_i$ that
  vanishes at all of these $z$.  Then $\deg F \geq 4$ by
  construction.

  By Propositions \ref{Roth-number} and \ref{Roth-function}, with $\epsilon = 1$, there
  is a nonzero constant $C_1$ such that
\[ \sum_{v_\fp (F(\varphi^{n-i}(\alpha))) > 0} \N_\fp > 
(\deg F - 3)h(\varphi^{n-i}(\alpha)) \geq h(\varphi^{n-i}(\alpha)) + C_1. \]
Applying Proposition \ref{old} with  $\delta = 1/( 2 d^i)$ and
using the fact that $h(\varphi^i(z)) \leq d^i h(z) +O(1)$ for
all $z \in K$, we see that there is a constant $C_2$ such that 
\[ \sum_{\fp \in Z} \N_\fp \leq \frac{1}{2} h_\varphi(\varphi^{n-i}(\alpha)) +
C_2\] where $Z$ is the set of primes $\fp$ such that
$\min(v_\fp(F(\varphi^{n-i}(\alpha))), v_\fp(\varphi^m(\alpha))) > 0$ for some positive
integer $m < n$.
Thus, when $h(\varphi^{n-i}(\alpha)) > 2(C_2 - C_1)$, we have
$\sum_{v_\fp (F(\varphi^{n-i}(\alpha))) > 0} \N_\fp > \sum_{\fp \in Z}
\N_\fp $ so there is a prime
$\fp$ such that $v_\fp(P_i(\varphi^{n-i}(\alpha))) > 0$ but
$v_\fp(\varphi^m(\alpha)) \leq 0$ for all $m < n$.  Now, writing
$\varphi^i(x) = F(x) R(x)/T(x)$, where $FR$ and $T$ are coprime, we
see that for all but finitely many $\fp$, we have $v_\fp(\varphi^i(z))
> 0$ whenever $v_\fp(F(z)) > 0$. Since $\lim_{n \to \infty}
h(\varphi^{n-i}(\alpha)) = \infty$ (by \eqref{can3}), we see then that for all but finitely
many $n$, there is a prime $\fp$ such that 
$v_\fp(\varphi^n(\alpha)) >
0$ and $v_\fp(\varphi^m(\alpha)) \leq 0$ for all $1 \leq m < n$.  
\end{proof}

Theorem~\ref{main2} is proved in the same manner as
Theorem~\ref{main1}.   The only significant difference is that we use
a square-free factor $F$ of $P_i$, which is possible because $\varphi$
is not dynamically ramified.

\begin{proof}
  There is an $i$ such that $P_i$ has a factor $F$ of degree 8 (see
  Remark \ref{deg}) or more such that every root $\beta_j$ of $F$ is
  non-periodic, satisfies $\varphi^\ell(\beta_j) \not= 0$ for $\ell =
  0, \dots,i-1$, and has multiplicity 1 as a root of $P_i$.  To see
  this note that since $\varphi$ is not dynamically ramified are
  finitely many points $\gamma$ such that $\varphi^n(\gamma) = \beta$
  and $e_{\varphi^n}(\gamma/0) = 1$ for some $n$, where
  $e_{\varphi^n}(\gamma/0)$ is the ramification index of $\varphi^n$
  at $\gamma$ over 0.  . Thus, we may choose such a $\gamma$ that is
  not periodic and which is not in forward orbit of any ramification
  points or the point at infinity.  Then $\varphi^{-3}(\gamma)$
  contains at least 8 points (since $d \geq 2$) in
  $\varphi^{-(n+3)}(0)$ none of which are ramification points of
  $\varphi^{n+3}$.  None of this points can be periodic since $\gamma$
  is not periodic.  Let $i$ be the smallest $i$ such that
  $\varphi^i(z) = 0$ for these points $z$ (this $i$ is the same for
  all of them since they are all inverse images of the same
  non-periodic point), and let $F \in K[x]$ be a factor of $P_i$ that
  vanishes at all of these $z$.  Then $\deg F \geq 8$ by construction.

  Applying Roth-$abc$ to $F$ with $\epsilon = 1$, we obtain
\[ \sum_{v_\fp (F(\varphi^{n-i}(\alpha))) > 0} \N_\fp > 
(\deg F - 3) h(\varphi^{n-i}(\alpha)) + C_3\]
for some constant $C_3$, depending only on $F$.   
Since 
\[ \sum_{v_\fp (F(\varphi^{n-i}(\alpha))) > 0} v_\fp
(F(\varphi^{n-i}(\alpha))) \N_\fp \leq (\deg F)
h(\varphi^{n-i}(\alpha)) +O(1),\]
 we
see that there is a constant $C_4$ such that
\[ \sum_{v_\fp (F(\varphi^{n-i}(\alpha))) \geq 2} \N_\fp > \frac{\deg
  F}{2} h(\varphi^{n-i}(\alpha)) + C_4.\] Since $\deg F \geq 8$, we
have $(\deg F)/2 - 3 \geq 1$, so there is a constant $C_5$ such that
\[ \sum_{v_\fp (F(\varphi^{n-i}(\alpha))) =1} \N_\fp > 
h(\varphi^{n-i}(\alpha)) + C_5.\]
Applying Theorem \ref{old} with  $\delta = 1/(2 d^i)$ and
using the fact that $h(\varphi^i(z)) \leq d^i h(z) +O(1)$ for
all $z \in K$, we see that there is a constant $C_6$ such that 
\[ \sum_{\fp \in Z} \N_\fp \leq \frac{1}{2} h_\varphi(\varphi^n(\alpha)) +
C_6 \] where $Z$ is the set of primes $\fp$ such that
$\min(v_\fp(F(\varphi^{n-i}(\alpha)), v_\fp(\varphi^m(\alpha)))) > 0$ for $1 \leq m < n$.
Thus, when $h(\varphi^{n-i}(\alpha)) > 2(C_6 - C_4)$, there is prime
$\fp$ such that $v_\fp(F(\varphi^{n-i}(\alpha))) = 1$ but
$v_\fp(\varphi^m(\alpha)) \leq 0$ for all $m < n$. Now, writing
$\varphi^i(x) = F(x) R(x)/T(x)$, where $F$, $R$, and $T$ are pairwise coprime, we
see that for all but finitely many $\fp$, we have $v_\fp(\varphi^i(z))
= v_\fp(F)$ whenever $v_\fp(F(z)) > 0$. Since $\lim_{n \to \infty}
h(\varphi^{n-i}(\alpha)) = \infty$ (by \eqref{can3}), we see then that for all but finitely
many $n$, there is a prime $\fp$ such that 
$v_\fp(\varphi^n(\alpha)) = 1$
 and $v_\fp(\varphi^m(\alpha)) \leq 0$ for all $1 \leq m < n$. 
 
\end{proof}

\begin{remark}\label{deg}
In the proofs of Theorems \ref{main1} and \ref{main2}, the degree of
the polynomial $F$ could be taken as large as one likes.  Degrees 4
and 8, respectively, are simply convenient for the estimates.  We wish
to avoid the point at infinity so that we can take a polynomial $F(x)$
that vanishes at all of the points (without introducing homogenous
coordinates).  The reason we do not take $\deg F$ to be exactly 4 or 8
is that doing so might require passing to a finite extension of $K$,  and
we do not wish to assume the $abc$-conjecture for extensions of $K$
when $K$ is a number field.  
\end{remark}

\begin{remark}\label{dyn}
When $\varphi$ is dynamically ramified, there are at most
finitely many $\fp$ that appear as square-free factors of any
$\varphi^n(\alpha)$.  If $\varphi$ is dynamically ramified, then there are at most
finitely many polynomials that appear as factors of any $P_n$, where
$P_n$ is the numerator of $\varphi$.  Thus, for any $\alpha$, there
are only finitely many $\fp$ such that $v_\fp(\varphi^n(\alpha)) = 1$
for some $n$.  Thus, the conclusion of Theorem~\ref{main2} will never
hold for a dynamically ramified rational function.  
\end{remark}

\section{An application to iterated Galois groups}\label{application}
 
Our original motivation for the problem of square-free primitive
divisors comes from the study of Galois groups of iterates of
polynomials, that is, Galois groups of splitting fields of $f^m(x)$
for $f$ a polynomial.  Odoni \cite{Odoni, Odoni2} calculated these groups for
``generic polynomials'' and for the specific polynomial $x^2 + 1$.
Stoll \cite{Stoll} later calculated them for
polynomials of the form $x^2 + a$, where $a$ is a positive integer congruent to 1 or
2 modulo 4.   In particular, Stoll defines $\Omega_{n,a}$ to be the splitting field
of $f_a^n(x)$ for $f_a(x) = x^2 + a$ and shows that if $a$ is a positive integer congruent to 1 or
2 modulo 4, then 
\begin{equation}\label{open}
\text{$[\Omega_{n+1,a}: \Omega_{n,a}] = 2^{2^n}$ for all
$n \geq 0$.}
\end{equation}
This allows for a completely explicit description of $\Gal(\Omega_{m,a} / \bQ)$ for any $m$ in terms of an inductive wreath
product structure.  Stoll notes that \eqref{open} is not
true for $f_a(x) = x^2+a$ when $a$ is an integer of the form $-b^2 - 1$ for $b$
a positive integer, since in this case one has $[\Omega_{2,a}:
\Omega_{1,a}] = 2$.

\begin{proposition}\label{full}
Suppose that the $abc$-conjecture for $\bQ$ holds.  Let $a
\not= -2$ be an integer such that $-a$ is not a perfect square in
$\bZ$.  Then, with notation above, we have
\begin{equation}\label{full2}
 [\Omega_{n+1,a}: \Omega_{n,a}] = 2^{2^n}
\end{equation} 
for all but finitely many natural numbers $n$.
\end{proposition}
\begin{proof}
By \cite[Lemma 1.6]{Stoll}, we have \eqref{full2} whenever
$f_a^{n+1}(0)$ is not a square in $\Omega{n,a}$.  
A simple calculation with discriminants (see \cite[Lemma 3.1]{Odoni} or \cite[Lemma
4.10]{Jones2}, for example) shows that that  $\Disc f^m(x) = 2^{2^m} \cdot
\Disc(f^{m-1}(x)) \cdot f^m(0)$. Hence, by induction we see that $\Omega_{n,a}$ is
unramified away from primes dividing $2 \prod_{i=1}^n f^i(0)$.   Since $f_a(x)$ is
not dynamically ramified and 0 is not preperiodic for $f_a$, we may
apply Theorem~\ref{main2} and conclude that for all but finitely many
$n$, there is an prime $\fp \not= 2$ such that $v_\fp(f^{n+1}(0)) = 1$ and
$v_\fp(f^m(0)) = 0$ for all $1 \leq m < n +1$ (note that a negative
valuation for any $f^m(0)$ s not possible since $a$ is an integer).
Thus, for all but finitely many $n$, we see that $f_a^{n+1}(0)$ is not
a square in $\Omega_{n,a}$, which finishes our proof.
\end{proof}

\begin{remark}

  Using arguments from \cite{Jones2}, one can show
  Proposition~\ref{full} holds for any $a \not = 0, -1, -2$; the
  proof, however, becomes more complicated.  One can also use similar
  arguments to show that the $abc$-conjecture implies \cite[Conjecture
  1.1]{JonesM}.  We plan to return to this problem in more generality
  in future work.
\end{remark}
\begin{remark}
One might also ask if Proposition~\ref{full} holds for function
  fields in characteristic 0.  We let $K$ be a function field of
  characteristic 0, $a \in K$ be non-constant, $f_a(x) = x^2 + a$, and
  $\Omega_{n,a}$ be the splitting field of $f_a^n(x)$ over $K$.  If $K
  = k(a)$, for $k$, the field of constants of $K$, then a
  specialization argument (from $a$ to an positive integer
  congruent to 1 or 2 mod 4) shows that
  $[\Omega_{n+1,a}: \Omega_{n,a}] = 2^{2^n}$ for all $n$.  Since
  $[K:k(a)]$ is finite for any non-constant $a \in K$, it follows that
  for any non-constant $a$ in a function field $K$ of characteristic 0,
  one has $[\Omega_{n+1,a}: \Omega_{n,a}] = 2^{2^n}$ for all but
  finitely many $n$.  
\end{remark}

Let $T$ be the binary rooted tree whose vertices at level $n$ are the
roots of $f^n(x)$.  Then Proposition~\ref{full} implies that the
natural image of $\Gal(\Qbar/\bQ)$ into $\Aut(T)$ has finite index in
$\Aut(T)$.  This can be interpreted as a dynamical analog of Serre's openness
theorem for Galois representations on torsion points of
elliptic curves \cite{Serre-Open} (see 
\cite{JB}).

% \bibliographystyle{amsalpha}
%\bibliographystyle{amsalpha}
%\bibliography{Drinbib}

\providecommand{\bysame}{\leavevmode\hbox to3em{\hrulefill}\thinspace}
\providecommand{\MR}{\relax\ifhmode\unskip\space\fi MR }
\providecommand{\MRhref}[2]{%
  \href{http://www.ams.org/mathscinet-getitem?mr=#1}{#2}
}
\providecommand{\href}[2]{#2}

\end{document}